\documentclass[a4paper,10pt]{article}
\usepackage[utf8]{inputenc}

\usepackage{enumerate}
\usepackage{amsthm}
\usepackage{amsmath}
\usepackage{amssymb}
\usepackage{mathtools}
\usepackage{bm}
\usepackage{tikz}
\usepackage{IEEEtrantools}
\usepackage{graphicx}
\usepackage{emptypage}
\usepackage{color}
\usepackage{changepage}

\newtheorem*{thm*}{Theorem}
\newtheorem{thm}{Theorem}[section]
\newtheorem{lemma}[thm]{Lemma}
\newtheorem{prop}[thm]{Proposition}
\newtheorem{corollary}[thm]{Corollary}

\theoremstyle{definition}
\newtheorem{mydef}{Definition}

\theoremstyle{remark}
\newtheorem{remark}{Remark}

\newcommand{\ie}{\emph{i.e.}}
\newcommand{\hsmallspace}{\hspace{0.5cm}}
\newcommand{\vsmallspace}{\vspace{0.5cm}}
\newcommand{\R}{\mathbb{R}}

\newcommand{\N}{\mathbb{N}}

\newcommand{\rnd}{\R^{Nd}}
\newcommand{\jbar}{\bar{\jmath}}
\newcommand{\gra}[1]{\left \{ #1 \right \}}
\newcommand{\st}{\left | \right.}
\newcommand{\abs}[1]{\left| #1 \right|}

\newcommand{\de}{\mathop{}\!\mathrm{d}}
\DeclareMathOperator{\supp}{supp}

\title{Marginals with finite repulsive cost}
\author{U. Bindini}
\date{}

\begin{document}

\maketitle

\begin{abstract}
We consider a multimarginal transport problem with repulsive cost, where the marginals are all equal to a fixed probability $\rho \in \mathcal{P}(\R^d)$. We prove that, if the concentration of $\rho$ is less than $1/N$, then the problem has a solution of finite cost. The result is sharp, in the sense that there exists $\rho$ with concentration $1/N$ for which the cost is infinite.
\end{abstract}

\emph{2010 Mathematics Subject Classification:} 49J10, 49K10.

\section{Introduction}

Consider a system of $N$ unitary-charged particles of negligible mass under the effect of the Coulomb force. We may describe the stationary states using a wave-function $\psi(x_1, \dotsc, x_N)$, where $x_j \in \R^3$; via the Born interpretation, $\abs{\psi(x_1, \dotsc, x_N)}^2$ may be viewed as the density of the probability that the particles occupy the positions $x_1, \dotsc, x_N$, and it is symmetric, since the particles are indistinguishable.

When the semi-classical limit is considered, as already proved in \cite{bindini2017, cotar2013, cotar2017, lewin2017}, the stationary states reach the minimum of potential energy, \ie,
\begin{equation} \label{pot-min}
 V_0 = \min_{\psi} V(\psi) = \min \int_{\R^{3N}} c(x_1, \dotsc, x_N) \abs{\psi(x_1, \dotsc, x_N)}^2 \de x_1 \dotsm \de x_N, 
\end{equation}
where $c$ is the Coulomb (potential) cost function $c \colon (\R^3)^N \to \R$ defined as
\[
 c(x_1,\dotsc, x_N) = \sum_{1 \leq i < j \leq N} \frac{1}{\abs{x_i - x_j}}.
\]

This can also be viewed as the exchange correlation functional linking the Kohn-Sham to the Hohenberg-Kohn approach, see for instance \cite{gori2009density}.

Given any wave-function $\psi$, define its single-particle density as
\[
 \rho^{\psi}(x) = \int_{\R^{3(N-1)}} \abs{\psi(x, x_2, \dotsc, x_N)}^2 \de x_2 \dotsm \de x_N,
\]
which is quite natural from the physical point of view, since the charge density is a fundamental quantum-mechanical observable.

It is a well-known result by Lieb \cite{lieb1983} (see also Levy \cite{levy1979}) that the set of all possible marginal densities is
\[
 \mathcal{R} = \gra{\rho \in L^1(\R^d) \st \rho \geq 0, \sqrt{\rho} \in H^1(\R^d), \int_{\R^d} \rho(x) dx = 1}.
\]

One may thus consider
\[
 C(\rho) = \min \gra{ \int_{\R^{3N}} c(x_1, \dotsc, x_N) \abs{\psi(x_1, \dotsc, x_N)}^2 \de x_1 \dotsm \de x_N \st \rho^{\psi} = \rho },
\]
and factorize the original minimum problem \eqref{pot-min} as
\[
 V_0 = \min_{\rho \in \mathcal{R}} \min_{\rho^\psi = \rho} V(\psi) = \min_{\rho} C(\rho).
\]

This is a well known approach, which dates back to Thomas and Fermi, and was later revised by Hohenberg and Kohn \cite{hohenberg1964}, Levy \cite{levy1979} and Lieb \cite{lieb1983}, whose questions are still sources of ideas for this field.

\vsmallspace

In this paper, firstly we generalize the physical dimension $d = 3$ to any $d \geq 1$. Moreover, we adopt a measure-theoretic approach: instead of considering wave-functions, we set the problem for every probability over $(\R^d)^N$ and formulate the corresponding relaxed minimum problem
\[
 \mathcal{C}(P) = \min \int_{(\R^d)^N} c(x_1,\dotsc, x_N) \de P(x_1, \dotsc, x_N),
\]
where $P \in \mathcal{P}((\R^d)^N)$ is a probability measure. In this fashion, the single-particle density constraint gives rise to a multi-marginal optimal transport problem of the form
\begin{align}
\begin{split} \label{multi-problem}
 C(\rho) = \inf \bigg \{ &\int_{(\R^d)^N} c(x_1,\dotsc, x_N) \de P(x_1, \dotsc, x_N) \colon \\
 & P \in \mathcal{P}((\R^d)^N), \ \pi_{\#}^i P = \rho, i = 1, \dotsc, N \bigg \},
\end{split}
\end{align}
where $\rho$ is a fixed probability measure over $\R^d$, and $\pi^i$ is the projection over the $i$-th factor of $(\R^d)^N$. It is a simple and well known observation that the infimum \eqref{multi-problem} is equal to
\begin{align}
\begin{split} \label{sym-multi-problem}
 C(\rho) = \inf \bigg \{ &\int_{(\R^d)^N} c(x_1,\dotsc, x_N) \de P(x_1, \dotsc, x_N) \colon \\
 & P \in \mathcal{P}((\R^d)^N), P \text{ symmetric }, \pi_{\#}^i P = \rho, i = 1, \dotsc, N \bigg \}.
\end{split}
\end{align}

In order to give even a stronger result, we take as a cost function a general repulsive potential, as in the following

\begin{mydef} \label{repulsive-cost} A function $c \colon (\R^d)^N \to \R$ is a \emph{repulsive cost function} if it is of the form
\[
 c(x_1, \dotsc, x_N) = \sum_{1 \leq i < j \leq N} \frac{1}{\omega(\abs{x_i-x_j})}
\]
where $\omega \colon \R^{+} \to \R^{+}$ is continuous, strictly increasing, differentiable on $(0, +\infty)$, with $\omega(0) = 0$.
\end{mydef}

Although there are many works about this formulation, and the multi-marginal transport problem in general (see for instance \cite{buttazzo2016, colombo2013multimarginal, colombo2013equality, depascale2015, dimarino2017}), none of them gives a condition on $\rho$ which assures that the infimum in \eqref{sym-multi-problem} is finite. We found that the correct quantity to consider is the one given by the following

\begin{mydef}
 If $\rho \in \mathcal{P}(\R^d)$, the \emph{concentration} of $\rho$ is
  \[
  \mu(\rho) = \sup_{x \in \R^d} \rho(\gra{x}).
  \]
\end{mydef}

This allows us to state the main result:

\begin{thm} \label{main-thm}
 Let $c$ be a repulsive cost function, and $\rho \in \mathcal{P}(\R^d)$ with
 \begin{equation} \label{conc-condition}
  \mu(\rho) < \frac{1}{N}.
 \end{equation}
 
 Then the infimum in \eqref{sym-multi-problem} is finite.
\end{thm}

After this paper was already submitted, the author became aware of an independent work in preparation by F. Stra, S. Di Marino and M. Colombo about the same problem. The techniques are different and the second result, although not yet available in preprint form, seems to be closer in the approach to some arguments in \cite{buttazzo2016}.


\paragraph{Structure of the paper}
In Section \ref{notation} we give some notation, and regroup some definitions, constructions and results to be used later. In particular, we state and prove a simple but useful result about partitioning $\R^d$ into measurable sets with prescribed mass.

We then show in Section \ref{sharpness} that the condition \eqref{conc-condition} is sharp, \ie, given any repulsive cost function, there exists $\rho \in \mathcal{P}(\R^d)$ with $\mu(\rho) = 1/N$, and $C(\rho) = \infty$. The construction of this counterexample is explicit, but it is important to note that the marginal $\rho$ depends on the given cost function.

Finally we devote Sections \ref{zero-atoms} to \ref{final-section} to the proof of Theorem \ref{main-thm}. The construction is universal, in the following sense: given $\rho \in \mathcal{P}(\R^d)$ such that \eqref{conc-condition} holds, we exhibit a symmetric transport plan $P$ which has support outside the region
\[
 D_{\alpha} = \gra{(x_1, \dotsc, x_N) \in (\R^d)^N \st \exists i \neq j \text{ with } \abs{x_i-x_j} < \alpha}
\]
for some $\alpha > 0$. This implies that $C(P)$ is finite for any repulsive cost function.

\paragraph{Aknowledgements}
The author is grateful to prof. Luigi Ambrosio and prof. Luigi De Pascale for all their useful remarks, and wish to thank prof. Emmanuel Trélat for his suggestions.


\section{Notation and preliminary results} \label{notation}

In the following, $x, x_j$ denote elements of $\R^d$, and $X = (x_1, \dotsc, x_N)$ is an element of $(\R^d)^N = \R^{Nd}$. We also indicate with $B(x_j, r)$ a ball with center $x_j \in \R^d$ and radius $r > 0$. Where it is not specified, the integrals are extended to all the space; if $\tau$ is a measure over $\R^d$, we denote by $\abs{\tau}$ its total mass, \ie, 
\[
 \abs{\tau} = \int_{\R^d} \de \tau.
\]

We use the expression \emph{$N$-transport plan for the marginal $\rho$} to denote a probability measure $P \in \mathcal{P}(\R^{Nd})$ with all the marginals
equal to $\rho \in \mathcal{P}(\R^d)$.

If $P \in \mathcal{M}(\R^{Nd})$ is any measure, we define
\[
 P_{sym} = \frac{1}{N!} \sum_{s \in S_N} \phi^s_{\#} P,
\]
where $S_N$ is the premutation group over the elements $\gra{1, \dotsc, N}$, and $\phi^s \colon \R^{Nd} \to \R^{Nd}$ is the function $\phi^s(x_1, \dotsc, x_N) = (x_{s(1)}, \dotsc, x_{s(N)})$. Note that $P_{sym}$ is a symmetric measure; moreover, if $P$ is a probability measure, then also $P_{sym}$ is a probability measure.

\begin{lemma} \label{sym-marginals}
 Let $P \in \mathcal{M}(\R^{Nd})$. Then $P_{sym}$ has marginals equal to
 \[
  \frac{1}{N} \sum_{j = 1}^N \pi^{j}_{\#} P
 \]
\end{lemma}

\begin{proof}
 Since $P_{sym}$ is symmetric, me may calculate its first marginal:
 \begin{align*}
  \pi^1_{\#} P_{sym} &= \pi^1_{\#} \left( \frac{1}{N!} \sum_{s \in S_N} \phi^s_{\#} P \right) = \frac{1}{N!} \sum_{s \in S_N} \pi^1_{\#} (\phi^s_{\#} P) \\
  &= \frac{1}{N!} \sum_{s \in S_N} \pi^{s(1)}_{\#} P = \frac{1}{N} \sum_{j = 1}^N \pi^{j}_{\#} P,
 \end{align*}
 where the last equality is due to the fact that for every $j = 1, \dotsc, N$ there are exactly $(N-1)!$ permutations $s \in S_N$ such that $s(1) = j$. 
\end{proof}

For a symmetric probability $P \in \mathcal{P}(\R^{Nd})$ we will use the shortened notation $\pi(P)$ to denote its marginals $\pi^{j}_{\#} P$, which are all equal. 

If $\sigma_1, \dotsc, \sigma_N \in \mathcal{M}(\R^d)$, we define $\sigma_1 \otimes \dotsb \otimes \sigma_N \in \mathcal{M}(\R^{Nd})$ as the usual product measure. In similar fashion, if $Q \in \mathcal{M}(\R^{(N-1)d})$, $\sigma \in \mathcal{M}(\R^d)$ and $1 \leq j \leq N$, we define the measure $Q \otimes_j \sigma \in \mathcal{M}(\R^{Nd})$ as
\begin{equation} \label{tensor-j}
 \int_{\R^{Nd}} f \de (Q \otimes_j \sigma) = \int_{\R^{Nd}} f(x_1, \dotsc, x_N) \de \sigma(x_j) \de Q(x_1, \dotsc, \hat{x}_j, \dotsc, x_N)
\end{equation}
for every $f \in C_b(\R^{Nd})$.

\subsection{Partitions of non-atomic measures}

Let $\sigma \in \mathcal{M}(\R^d)$ be a finite non-atomic measure, and $b_1, \dotsc, b_k$ real positive numbers such that $b_1 + \dotsb + b_k = \abs{\sigma}$. We may want to write
\[
 \R^d = \bigcup_{j = 1}^k E_j,
\]
where the $E_j$'s are disjoint measurable sets with $\sigma(E_j) = b_j$. This is trivial if $d = 1$, since the cumulative distribution function $\phi_{\sigma}(t) = \sigma((-\infty, t))$ is continuous, and one may find the $E_j$'s as intervals. However, in higher dimension, the measure $\sigma$ might concentrate over $(d-1)$-dimensional surfaces, which makes the problem slightly more difficult. Therefore we present the following

\begin{prop} \label{slice}
 Let $\sigma \in \mathcal{M}(\R^d)$ be a finite non-atomic measure. Then there exists a direction $y \in \R^d \setminus \gra{0}$ such that $\sigma(H) = 0$ for all the affine hyperplanes $H$ such that $H \perp y$.
\end{prop}

In order to prove Proposition \ref{slice}, it is useful to present the following

\begin{lemma} \label{measure-lemma}
	Let $(X,\mu)$ be a measure space, with $\mu(X) < \infty$, and $\{E_i\}_{i \in I}$ a collection of measurable sets such that
	\begin{enumerate}
		\item $\mu(E_i) > 0$ for every $i \in I$;
		\item $\mu(E_i \cap E_j) = 0$ for every $i \neq j$.
	\end{enumerate}
	
	Then $I$ is countable. 
\end{lemma}

\begin{proof} Let $i_1, \dotsc, i_n$ be a finite set of indices. Then using the monotonicity of $\mu$ and the fact that $\mu(E_i \cap E_j) = 0$ if $i \neq j$,
	\[
	\mu(X) \geq \mu \left( \bigcup_{k=1}^n E_{i_k} \right) = \sum_{k=1}^n \mu(E_{i_k}).
	\]
	
	Hence we have that
	\[
	\sup \gra{ \sum_{j \in J} \mu(E_j) \st J \subset I, J \text{ finite} } \leq \mu(X) < \infty. 
	\]
	
	Since all the $\mu(E_i)$ are strictly positive numbers, this is possible only if $I$ is countable.
\end{proof}

Now we present the proof of Proposition \ref{slice}.

\begin{proof}
For $k = 0, 1, \dotsc, d-1$ we recall the definitions of the Grassmannian
\[
 \mathrm{Gr}(k, \R^d) = \gra{v \text{ linear subspace of } \R^d \st \dim v = k}
\]
and the affine Grassmannian
\[
 \mathrm{Graff}(k, \R^d) = \gra{w \text{ affine subspace of } \R^d \st \dim w = k}.
\]

Given $w \in \mathrm{Graff}(k, \R^d)$, we denote by $[w]$ the unique element of $\mathrm{Gr}(k, \R^d)$ parallel to $w$. If $S \subseteq \mathrm{Graff}(k, \R^d)$, we say that $S$ is \emph{full} if for every $v \in \mathrm{Gr}(k, \R^d)$ there exists $w \in S$ such that $[w] = v$.
 For every $k = 1, 2, \dotsc, d-1$ let $S^{k} \subseteq \mathrm{Graff}(k, \R^d)$ be the set
 \[
  S^{k} = \gra{w \in \mathrm{Graff}(k, \R^d) \st \sigma(w) > 0}.
 \]
 
 The goal is to prove that $S^{d-1}$ is \emph{not} full, while by hypothesis we know that $S^0 = \emptyset$, since $\sigma$ is non-atomic.
 
 The following key Lemma leads to the proof in a finite number of steps:
 
 \begin{lemma}
  Let $1 \leq k \leq d-1$. If $S^{k-1}$ is not full, then $S^{k}$ is not full.
\end{lemma}
 
 \begin{proof} Let $v \in \mathrm{Gr}(k-1, \R^d)$, such that for every $v' \in \mathrm{Graff}(k-1, \R^d)$ with $[v'] = v$ it holds $\sigma(v') = 0$. Consider the collection $W_{v} = \gra{w \in \mathrm{Graff}(k, \R^d) \st v \subseteq [w]}$. If $w,w' \in W_v$ are distinct, then $w \cap w' \subseteq v'$ for some $v' \in \mathrm{Graff}(k-1, \R^d)$ with $[v'] = v$, thus $\sigma(w \cap w') = 0$. Since the measure $\sigma$ is finite, because of Lemma \ref{measure-lemma} at most countably many elements $w \in W_v$ may have positive measure, which implies that $S^k$ is not full.
 \end{proof}

\end{proof}

\begin{corollary} \label{cor-partition-1}
 Given $b_1, \dotsc, b_k$ real positive numbers with $b_1 + \dotsb + b_k = \abs{\sigma}$, there exist measurable sets $E_1, \dotsc, E_k \subseteq \R^d$ such that
 \begin{enumerate}[(i)]
  \item The $E_j$'s form a partition of $\R^d$, \ie, 
  \[
   \R^d = \bigcup_{j = 1}^k E_j, \hsmallspace E_i \cap E_j = \emptyset \text{ if $i \neq j$;}
  \]
  \item $\sigma(E_j) = b_j$ for every $j = 1, \dotsc, k$.

 \end{enumerate}
\end{corollary} 

\begin{proof} Let $y \in \R^d \setminus \gra{0}$ given by Proposition \ref{slice}, and observe that the cumulative distribution function
\[
 F(t) = \sigma \left( \gra{x \in \R^d \st x \cdot y < t} \right)
\]
is continuous. Hence we may find $E_1, \dotsc, E_k$ each of the form
\[
 E_j = \gra{x \in \R^d \st t_j < x \cdot y \leq t_{j+1}}
\]
for suitable $-\infty = t_1 < t_2 < \dotsb < t_k < t_{k+1} = +\infty$, such that $\sigma(E_j) = b_j$.
\end{proof}

\begin{corollary} \label{cor-partition-2}
 Given $b_1, \dotsc, b_k$ non-negative numbers with $b_1 + \dotsb + b_k < \abs{\sigma}$, there exists measurable sets $E_0, E_1, \dotsc, E_k \subseteq \R^d$ such that
 \begin{enumerate}[(i)]
  \item The $E_j$'s form a partition of $\R^d$, \ie, 
  \[
   \R^d = \bigcup_{j = 0}^k E_j, \hsmallspace E_i \cap E_j = \emptyset \text{ if $i \neq j$;}
  \]
  \item $\sigma(E_j) = b_j$ for every $j = 1, \dotsc, k$;
  \item the distance between $E_i$ and $E_j$ is strictly positive if $i, j \geq 1$, $i \neq j$. 
 \end{enumerate}
\end{corollary}

\begin{proof} If $k = 1$ the results follows trivially by Corollary \ref{cor-partition-1} applied to $b_1, \abs{\sigma} - b_1$. If $k \geq 2$, define
\[
 \epsilon = \frac{\abs{\sigma} - b_1 - \dotsb - b_k}{k-1} > 0.
\]

As before, letting $y \in \R^d \setminus \gra{0}$ given by Proposition \ref{slice} and considering the corresponding cumulative distribution function, we may find $F_1, \dotsc, F_{2k-1}$ each of the form
\[
 F_j = \gra{x \in \R^d \st t_j < x \cdot y \leq t_{j+1}}
\]
for suitable $-\infty = t_1 < t_2 < \dotsb < t_{2k-1} < t_{2k} = +\infty$, such that
\begin{align*}
 \sigma(F_{2j-1}) &= b_j \hsmallspace \forall j = 1, \dotsc, k \\
 \sigma(F_{2j}) &= \epsilon \hsmallspace \forall j = 1, \dotsc, k-1
\end{align*}

Finally we define
\begin{align*}
 E_j &= F_{2j-1} \hsmallspace \forall j = 1, \dotsc, k \\
 E_0 &= \bigcup_{j = 1}^{k-1} F_{2j}.
\end{align*}

The properties (i), (ii) are immediate to check, while the distance between $E_i$ and $E_j$, for $i, j \geq 1$, $i \neq j$, is uniformly bounded from below by
\[
 \min \gra{t_{2j+1} - t_{2j} \st 1 \leq j \leq k-1} > 0.
\]

\end{proof}


\section{The condition \eqref{conc-condition} is sharp} \label{sharpness}

In this section we prove that the condition \eqref{conc-condition} is the best possible, \ie, given any repulsive cost function there exists $\rho \in \mathcal{P}(\R^d)$ with $\mu(\rho) = 1/N$ such that $C(\rho) = \infty$.

Fix $\omega$ as in Definition \ref{repulsive-cost}, and set
\[
 k = \int_{B(0,1)} \frac{\omega'(\abs{y})}{\abs{y}^{d-1}} \de y.
\]

Note that $k$ is a positive finite constant, depending only on $\omega$ and the dimension $d$. In fact, integrating in spherical coordinates,
\[
 k = \int_0^1 \frac{\omega'(r)}{r^{d-1}} \alpha_d r^{d-1} \de r = \alpha_d \omega(1), 
\]
where $\alpha_d$ is the $d$-dimensional volume of the unit ball $B(0,1) \subseteq \R^d$.

Now define a probability measure $\rho \in \mathcal{P}(\R^d)$ as
\begin{equation} \label{rho-def}
 \int_{\R^d} f \de \rho := \frac{1}{N} f(0) + \frac{N-1}{N} \int_{B(0,1)} f(x) \frac{\omega'(\abs{x})}{k\abs{x}^{d-1}} \de x \hsmallspace \forall f \in C_b(\R^d).
\end{equation}

This measure has an atom of mass $1/N$ in the origin, and is absolutely continuous on $\R^d \setminus \gra{0}$. Hence the concentration of $\rho$ is equal to $1/N$, even if for every ball $B$ around the origin one has $\rho(B) > 1/N$.

We want to prove that any symmetric transport plan with marginals $\rho$ has infinite cost. Let us consider, by contradiction, a symmetric plan $P$, with $\pi(P) = \rho$, such that
\[
 \int \sum_{1 \leq i < j \leq N} \frac{1}{\omega(\abs{x_i-x_j})} \de P(X) < \infty
\]

Then one would have the following geometric properties.

\begin{lemma} \label{hyperplanes}
\begin{enumerate}[(i)]
	\item $P(\gra{(x_1, \dotsc, x_N) : \exists i \neq j, x_i = x_j}) = 0;$
	\item $P$ is concentrated over the $N$ coordinate hyperplanes $\gra{x_j = 0}$, $j = 1, \dotsc, N$, \ie,
	\[
	 \supp(P) \subseteq E := \bigcup_{j = 1}^N \gra{x_j = 0}.
	\]
\end{enumerate}
\end{lemma}

\begin{proof}
	(i) Since $\omega(0) = 0$, recalling Definition \ref{repulsive-cost}, the cost function is identically equal to $+\infty$ in the region $\gra{(x_1, \dotsc, x_N) : \exists i \neq j, x_i = x_j}$. Therefore, since by assumption the cost of $P$ is finite, it must be
	\[
	P(\gra{(x_1, \dotsc, x_N) : \exists i \neq j, x_i = x_j}) = 0.
	\]
		
	(ii) Define
 \begin{align*}
  p_1 &= P(\gra{x_1 = 0}) \\
  p_2 &= P(\gra{x_1 = 0} \cap \gra{x_2 = 0}) \\
  &\vdots \\
  p_N &= P((0,\dotsc,0)).
 \end{align*}
 
 Note that $p_1 = P(\gra{x_1 = 0}) = \pi(P)(\gra{0}) = \rho(\gra{0}) = 1/N$. We claim that $p_2 = \dotsb = p_N = 0$. It suffices to prove that $p_2 = 0$, since by monotonicity of the measure $P$ we have $p_j \geq p_{j+1}$. Since $P$ has finite cost,
 \[
  \int_{\R^{Nd}} \frac{\de P}{\omega(\abs{x_1-x_2})}
 \]
 must be finite. However,
 \[
  \int_{\R^{Nd}} \frac{\de P}{\omega(\abs{x_1-x_2})} \geq \int_{\gra{x_1 = 0} \cap \gra{x_2 = 0}} \frac{\de P}{\omega(\abs{x_1-x_2})} = p_2 \int_{\R^{2d}} \frac{\delta_0(x_1) \delta_0(x_2)}{\omega(\abs{x_1-x_2})} \de x_1 \de x_2,
 \]
 and hence $p_2$ must be zero.
 
 By inclusion-exclusion we have
 \[
  P(E) = \sum_{j = 1}^N (-1)^{j+1} \binom{N}{j} p_j = Np_1 = 1,
 \]
 and hence $P$ is concentrated over $E$. 
\end{proof}

In view of Lemma \ref{hyperplanes}, letting $H_j = \{x_j = 0\}$ for $j = 1, \dotsc, N$,
\[P = \sum_{j = 1}^N P|_{H_j}.\]

For every $j = 1, \dotsc, N$ there exists a unique measure $Q_j$ over $\R^{(N-1)d}$ such that, recalling equation \eqref{tensor-j}, $P|_{H_j} = Q_j \otimes_j \delta_0$, with $Q_j(\R^{(N-1)d}) = \frac{1}{N}$. Since $P$ is symmetric, considering a permutation $s \in S_N$ with $s(j) = j$, it follows that $Q_j$ is symmetric; then, considering any permutation in $S_N$ we see that there exists a symmetric probability $Q$ over $\R^{(N-1)d}$ such that $Q_j = \frac{1}{N} Q$ for every $j = 1, \dotsc, N$, \ie,
\[P = \frac{1}{N} \sum_{j = 1}^N Q \otimes_j \delta_0.\] 

Projecting $P$ to its one-particle marginal and using the definition of $\rho$ in \eqref{rho-def}, we get that $\pi(Q)$ is absolutely continuous w.r.t. the Lebesgue measure, with
\[
 \frac{\de \pi(Q)}{\de \mathcal{L}^d} = \frac{\chi_{B(0,1)}(x)\omega'(x)}{k \abs{x}^{d-1}}.
\]

Here we get the contradiction, because
\begin{align*}
 \int c(X) dP(X) &\geq \frac{1}{N} \int \frac{1}{\omega(\abs{x_1-x_2})} \delta_{0} (x_1) \de x_1 \de Q(x_2, \dotsc, x_N) \\
 &= \frac{1}{N} \int \frac{1}{\omega(\abs{x_2})} \de Q(x_2, \dotsc, x_N) = \frac{1}{N} \int_{\R^d} \frac{1}{\omega(\abs{x})} \de \pi (Q)(x) \\
 &= \frac{1}{N} \int_{B(0,1)} \frac{\omega'(\abs{x})}{\omega(\abs{x})} \frac{1}{k \abs{x}^{d-1}} \de x = \frac{1}{N}\frac{\alpha_d}{k} \int_0^1 \frac{\omega'(r)}{\omega(r)} \de r = +\infty.
\end{align*}


\section{Non-atomic marginals} \label{zero-atoms}

This short section deals with the case where $\rho$ is non atomic, \ie, $\mu(\rho) = 0$. In this case the transport plan is given by an optimal transport map in Monge's fashion, which we proceed to construct.

Using Corollary \ref{cor-partition-1}, let $E_1, \dotsc, E_{2N}$ be a partition of $\R^d$ such that
\[
 \rho(E_j) = \frac{1}{2N} \hsmallspace \forall j = 1, \dotsc, 2N.
\]

Next we take a measurable function $\phi \colon \R^d \to \R^d$, preserving the measure $\rho$ and defined locally such that
\begin{align*}
  \phi(E_j) &= E_{j+2} \hsmallspace \forall j = 1, \dotsc, N-2 \\
  \phi (E_{2N-1}) &= E_1 \\
  \phi (E_{2N}) &= E_2.
\end{align*}

The behaviour of $\phi$ on the hyperplanes which separate the $E_j$'s is arbitrary, since they form a $\rho$-null set. Note that $\abs{x - \phi(x)}$ is uniformy bounded from below by some constant $\gamma > 0$, as is clear by the construction of the $E_j$'s (see the proof of Corollary \ref{cor-partition-1}). A transport plan $P$ of finite cost is now defined for every $f \in C_b(\R^{Nd})$ by
\[
 \int_{\R^{Nd}} f \de P = \int_{\R^{Nd}} f(x, \phi(x), \dotsc, \phi^{N-1}(x)) \de \rho(x),
\]
since
\[
 \int_{\rnd} c \de P = \binom{N}{2} \int_{\R^d} \frac{1}{\omega(\abs{x - \phi(x)})} \de \rho(x) \leq \binom{N}{2} \frac{1}{\omega(\gamma)}.
\]


\section{Marginals with a finite number of atoms} \label{finite-atoms} 

This section constitutes the core of the proof, as we deal with measures of general form with an arbitrary (but finite) number of atoms. Throughout this and the next Section we assume that the marginal $\rho$ fulfills the condition \eqref{conc-condition}.

\subsection{The number of atoms is less than or equal to $N$}

Note that, if the number of atoms is at most $N$, then $\rho$ must have a non-atomic part $\sigma$, due to the condition \eqref{conc-condition}. From here on we consider
\[
 \rho = \sigma + \sum_{i = 1}^k b_i \delta_{x_i},
\]
where $b_1 \geq b_2 \geq \dotsb \geq b_k > 0$.

We begin with the following

\begin{mydef} A \emph{partition} of $\sigma$ of level $k \leq N$ subordinate to $(x_1, \dotsc, x_k;$ $b_1, \dotsc, b_k)$ is
\[
 \sigma = \tau + \sum_{i = 1}^k \sum_{h = i+1}^N \sigma^i_h,
\]
where:
\begin{enumerate}[(i)]
 \item $\tau, \sigma^i_h$ are non-atomic measures;
 \item for every $i$ and every $h \neq k$, the distance between $\supp \sigma^i_h$ and $\supp \sigma^i_k$ is strictly positive;
 \item for every $i,h$, if $j \leq i$ then $x_j$ has a strictly positive distance from $\supp {\sigma^i_h}$; 
 \item for every $i,h$, $\abs{\sigma^i_h} = b_i$, and $\abs{\tau} > 0$.
\end{enumerate}
 
\end{mydef}

Note that such a partition may only exists if
\begin{equation} \label{partition-condition}
 \abs{\sigma} > \sum_{i = 1}^k (N-i) b_i.
\end{equation}
 
On the other hand, the following Lemma proves that the condition \eqref{partition-condition} is also sufficient to get a partition of $\sigma$.

\begin{lemma} \label{existence-partition} Let $(b_1, \dotsc, b_k)$ with $k \leq N$, and
\[
 \abs{\sigma} > \sum_{i = 1}^k (N-i) b_i.
\]

Then there exists a partition of $\sigma$ subordinate to $(x_1, \dotsc, x_k; b_1, \dotsc, b_k)$.
\end{lemma}

\begin{proof}
 Fix $(x_1, \dotsc, x_k)$ and for every $\varepsilon > 0$ define
 \[
  A_\varepsilon = \bigcup_{j = 1}^k B(x_j, \varepsilon).
 \]
 and $\sigma_{\varepsilon} = \sigma \chi_{A_\varepsilon}$. Then take $\varepsilon$ small enough such that
 \begin{equation} \label{existence-condition-eq1}
  \abs{\sigma - \sigma_{\varepsilon}} > \sum_{i = 1}^k (N-i)b_i,
 \end{equation}
 which is possibile because $\mu(\sigma) = 0$ ($\sigma$ has concentration zero), and hence $\abs{\sigma_\varepsilon} \to 0$ as $\varepsilon \to 0$. Due to Corollary \ref{cor-partition-2}, the set $\R^d \setminus A_{\varepsilon}$ may be partitioned as
 \[
  \R^d \setminus A_{\varepsilon} = \left( \bigcup_{i = 1}^k \bigcup_{h = i+1}^N E^i_h \right) \cup E,
 \]
 with $\sigma(E^i_h) = b_i$, and $\mathrm{dist}(E^i_h,E^i_k)$ is uniformly bounded from below.
 
 Finally define $\sigma^i_h = \sigma\chi_{E^i_h}$, $\tau = \sigma_\varepsilon + \sigma\chi_{E}$.

\end{proof}

\begin{prop} \label{partition-construction} Suppose that $k \leq N$ and $(b_1, \dotsc, b_k)$ are such that
\begin{equation} \label{existence-condition-eq2}
 \abs{\sigma} > N b_1 - \sum_{j = 1}^k b_j.
\end{equation}

Then there exists a transport plan of finite cost with marginals
\[
 \sigma + \sum_{j = 1}^k b_j \delta_{x_j}.
\]
\end{prop}

\begin{proof} In order to simplify the notation, set $b_{k+1} = 0$. First of all we shall fix a partition of $\sigma$ subordinate to $(x_1, \dotsc, x_k; b_1-b_2, \dotsc, b_{k-1}-b_k, b_k)$. To do this we apply Lemma \ref{partition-condition}, since
\begin{align*}
 \sum_{i = 1}^{k-1} (N-i)(b_{i} - b_{i+1}) + (N-k)b_k &= (N-1)b_1 - \sum_{i = 2}^k b_i < \abs{\sigma}.
\end{align*}

Next we define the measures $\lambda_i = \delta_{x_1} \otimes \dotsb \otimes \delta_{x_i} \otimes \sigma^i_{i+1} \otimes \dotsb \otimes \sigma^i_N \in \mathcal{M}(\R^{Nd})$. Let us calculate the marginals of $\lambda_i$: since $\abs{\sigma^i_h} = b_i - b_{i+1}$ for all $h = i+1, \dotsc, N$, we get
\[
 \pi^{j}_{\#} \lambda_i =
 \begin{cases}
  (b_i - b_{i+1})^{N-i} \delta_{x_j} & \text{if $0 \leq j \leq i$} \\
  (b_i - b_{i+1})^{N-i-1} \sigma^i_j & \text{if $i+1 \leq j \leq N$.}
 \end{cases}
\]

Let us define, for $i = 1, \dotsc, k$, the measure
\[
 P_i = \frac{N}{(b_i - b_{i+1})^{N-i-1}} (\lambda_i)_{sym},
\]
where $P_i = 0$ if $b_i = b_{i+1}$. By Lemma \ref{sym-marginals}, the marginals of $P_i$ are equal to
\[
 \pi(P_i) = \frac{1}{(b_i - b_{i+1})^{N-i-1}} \sum_{j = 0}^N \pi^j_{\#} \lambda_i = \sum_{j = 1}^i (b_i - b_{i+1}) \delta_{x_j} + \sum_{h = i+1}^N \sigma^i_h,
\]
so that
\[
 \sum_{i = 1}^k \pi(P_i) = \sum_{j = 1}^k b_j \delta_{x_j} + \sum_{i = 1}^k \sum_{h = i+1}^N \sigma^i_h
\]

It suffices now to take any symmetric transport plan $P_\tau$ of finite cost with marginals $\tau$, given by the result of Section \ref{zero-atoms}, and finally set
\[
 P = P_\tau + \sum_{i = 1}^k P_i.
\]
\end{proof}

As a corollary we obtain

\begin{thm} \label{N-or-less-atoms} If $\rho$ has $k \leq N$ atoms, then there exists a transport plan of finite cost.
 
\end{thm}

\begin{proof} Let
\[
 \rho = \sigma + \sum_{j=1}^k b_j \delta_{x_j}.
\]

Note that, since $b_1 < 1/N$,
\[
 \abs{\sigma} = 1 - \sum_{j = 1}^k b_j > Nb_1 - \sum_{j = 1}^k b_j,
\]
hence we may apply Proposition \ref{partition-construction} to conclude.
 
\end{proof}

\subsection{The number of atoms is greater than $N$}

Here we deal with the much more difficult situation in which $\rho$ has $N+1$ or more atoms, \ie,
\[\rho = \sigma + \sum_{j = 1}^k b_j \delta_{x_j}\]
with $k \geq N+1$ and as before $b_1 \geq b_2 \geq \dotsb \geq b_k > 0$. Note that in this case it might happen that $\sigma = 0$.

The main point is to use a double induction on the dimension $N$ and the number of atoms $k$, as will be clear in Proposition \ref{mega-prop}. The following lemma is a simple numerical trick needed for the inductive step in Proposition \ref{mega-prop}.

\begin{lemma} \label{t-lemma} Let $(b_1, \dotsc, b_k)$ with $k \geq N+2$ and
 \begin{equation} \label{main-condition}
 (N-1)b_1 \leq \sum_{j = 2} ^ k b_j.  
\end{equation}

Then there exist $t_2, \dotsc, t_k$ such that
\begin{enumerate}[(i)]
 \item $t_2 + \dotsb + t_k = (N-1)b_1$;
 \item for every $j= 2, \dotsc, k$, $0 \leq t_j \leq b_j$, and moreover
 \[
  t_2 \geq \dotsb \geq t_k.
 \]
 \[
  b_2 - t_2 \geq b_3 - t_3 \geq \dotsb \geq b_{k} - t_{k},
 \]
 \item \[
        (N-2) t_2 \leq \sum_{j = 3}^k t_j;
       \]
 \item \[
        (N-1) (b_2-t_2 ) \leq \sum_{j = 3}^k (b_j-t_j).
       \]
    
\end{enumerate}

\end{lemma}

\begin{proof} For $j=2, \dotsc, k$ define
\[
 p_j = \sum_{h = j}^k b_j,
\]
and let $\jbar$ be the least $j \geq 2$ such that $(N-j+2)b_j \leq p_j$; note that $j = N+2$ works --- hence $\jbar \leq N+2$. Define
\begin{align*}
 t_j &= b_j - \frac{p_2 - (N-1)b_1}{N} & \text{for $j = 2, \dotsc, \jbar-1$,} \\
 t_j &= b_j - \frac{b_j}{p_{\jbar}} \frac{p_2 - (N-1)b_1}{N} (N-\jbar+2) & \text{for $j = \jbar, \dotsc, k$.}
\end{align*}

Next we prove that this choice fulfills the conditions (i)-(iv).

\paragraph{Proof of (i)} 
\begin{align*}
 \sum_{j = 2}^k t_j &= p_2 - \frac{p_2 - (N-1)b_1}{N}(\jbar-2) - \frac{p_2 - (N-1)b_1}{N} (N-\jbar+2) \\
 &= p_2 \left( 1 - \frac{\jbar-2}{N} - \frac{N-\jbar+2}{N} \right) + (N-1)b_1 \left( \frac{\jbar-2}{N} + \frac{N-\jbar+2}{N} \right) \\
 &= (N-1)b_1.
\end{align*}

\paragraph{Proof of (ii)} In view of the fact that $(N-1)b_1 \leq p_2$ and $\jbar \leq N+2$, it is clear that $t_j \leq b_j$. If $j < \jbar$ we have $(N-j+2)b_j > p_j$, and hence
\[
 p_2 = b_2 + \dotsb + b_{j-1} + p_j < (j-2)b_1 + (N-j+2)b_j.
\]
Thus, since $2 \leq j \leq N+1$, 
\begin{align*}
 t_j &= \frac{Nb_j - p_2 + (N-1)b_1}{N} > \frac{Nb_j - (N-j+2)b_j - (j-2)b_1 + (N-1)b_1}{N} \\
 &= \frac{(j-2)b_j + (N-j+1)b_1}{N} \geq 0.
\end{align*}

To show that $t_j \geq 0$ for $j \geq \jbar$, we must prove $[p_2-(N-1)b_1](N-\jbar+2) \leq Np_{\jbar}$, which is trivial if $\jbar = N-2$. Otherwise, it is equivalent to
\[
 -(\jbar-2)[p_2 - (N-1)b_1] + N[b_2 + \dotsb + b_{\jbar-1} - (N-1)b_1] \leq 0.
\]
Since $2 \leq \jbar \leq N+1$, the first term is negative and $b_2 + \dotsb + b_{\jbar-1} - (N-1)b_1 \leq -(N-\jbar+1)b_1 \leq 0$.

\vsmallspace

Using the fact that $b_2 \geq \dotsb \geq b_k$, it is easy to see that $t_2 \geq \dotsb \geq \dotsb t_{\jbar-1}$ and $t_{\jbar} \geq \dotsb \geq t_k$ --- note that for $j \geq \jbar$ we have $t_j = \alpha b_j$, for some $0\leq \alpha \leq 1$. As for the remaining inequality,
\[
 t_{\jbar-1} \geq t_{\jbar} \iff b_{\jbar-1} - b_{\jbar} \geq \frac{p_2 - (N-1)b_1}{Np_{\jbar}} [p_{\jbar} -(N-\jbar+2)b_{\jbar}],
\]
we already proved
\[
 \frac{p_2 - (N-1)b_1}{Np_{\jbar}} \leq \frac{1}{N-\jbar+2};
\]
moreover, by definition of $\jbar$, we have $(N-\jbar+3)b_{\jbar-1} > p_{\jbar-1}$, or equivalently $(N-\jbar+2)b_{\jbar-1} > p_{\jbar}$. Thus
\[
 \frac{p_2 - (N-1)b_1}{Np_{\jbar}} [p_{\jbar} -(N-\jbar+2)b_{\jbar}] \leq \frac{p_{\jbar}}{N-\jbar+2} - b_{\jbar} < b_{\jbar-1} - b_{\jbar},
\]
as wanted.

It is left to show that $b_2-t_2 \geq \dotsb \geq b_k-t_k$. It is trivial to check that $b_2-t_2 = \dotsb = b_{\jbar-1} - t_{\jbar-1}$, and $b_{\jbar} - t_{\jbar} \geq \dotsb \geq b_k - t_k$ using $b_{\jbar} \geq \dotsb \geq b_k$ as before. Finally,
\[
 b_{\jbar-1}-t_{\jbar-1} \geq b_{\jbar} - t_{\jbar} \iff \frac{p_2-(N-1)b_1}{N} \geq \frac{b_{\jbar}}{p_{\jbar}} \frac{p_2-(N-1)b_1}{N} (N-\jbar+2),
\]
which is true since $(N-\jbar+2)b_{\jbar} \leq p_{\jbar}$ and $p_2-(N-1)b_1 \geq 0$.

\paragraph{Proof of (iii)} The thesis is equivalent to
\[
 (N-1) t_2 \leq \sum_{j = 2}^k t_j \iff (N-1)t_2 \leq (N-1)b_1,
\]
and this is implied by $t_2 \leq b_2 \leq b_1$.

\paragraph{Proof of (iv)} The thesis is equivalent to
\[
 N(b_2-t_2) \leq p_2 - (N-1)b_1,
\]
which is in fact an equality (see the definition of $t_2$).

\end{proof}

We are ready to present the main result of this Section, which provides a transport plan of finite cost under an additional hypothesis on the tuple $(b_1, \dotsc, b_k)$. The result is peculiar for the fact that it does not involve the non-atomic part of the measure -- it is in fact a general discrete construction to get a purely atomic symmetric measure having fixed purely atomic marginals.

\begin{prop} \label{mega-prop} Let $k > N$ and $(b_1, \dotsc, b_k)$ with
	\begin{equation} \label{mega-prop-condition}
	(N-1)b_1 \leq b_2 + \dotsb + b_k.
	\end{equation}
	
	Then for every $x_1, \dotsc, x_k \in \R^d$ distinct, there exists a symmetric transport plan of finite cost with marginals $\rho = b_1 \delta_{x_1} + \dotsb + b_k \delta_{x_k}$.
\end{prop}

\begin{proof} For every pair of positive integers $(N,k)$, with $k > N$, let $\mathfrak{P}(N,k)$ be the following proposition:

\vsmallspace

\begin{adjustwidth}{1cm}{1cm}
 Let $(x_1,\dotsc, x_k; b_1, \dotsc, b_k)$ with $(N-1)b_1 \leq b_2 + \dotsb + b_k$. Then for every $(x_1, \dotsc, x_k)$ there exists a symmetric $N$-transport plan of finite cost with marginals $b_1 \delta_{x_1} + \dotsb + b_k \delta_{x_k}$.
\end{adjustwidth}

\vsmallspace

We will prove $\mathfrak{P}(N,k)$ by double induction, in the following way: first we prove $\mathfrak{P}(1,k)$ for every $k$ and $\mathfrak{P}(N,N+1)$ for every $N$. Then we prove
\[
 \mathfrak{P}(N-1,k) \wedge \mathfrak{P}(N,k-1) \implies \mathfrak{P}(N,k).
\]

\paragraph{Proof of $\mathfrak{P}(1,k)$} This is trivial: simply take $b_1 \delta_{x_1} + \dotsb + b_k \delta_{x_k}$ as a ``transport plan''.
 
\paragraph{Proof of $\mathfrak{P}(N,N+1)$} Let us denote by $A_N$ the $(N+1)\times (N+1)$ matrix
\[
 A_N = \begin{pmatrix} 0 & 1 & \cdots & 1 \\ 1 & 0 & \cdots & 1 \\ \vdots & & \ddots & \\ 1 & \cdots & 1 & 0 \end{pmatrix},
\]
whose inverse is
\[
 A_N^{-1} = \frac{1}{N} \begin{pmatrix} -(N-1) & 1 & \cdots & 1 \\ 1 & -(N-1) & \cdots & 1 \\ \vdots & & \ddots & \\ 1 & \cdots & 1 & -(N-1) \end{pmatrix}
\]

Define also the following $(N+1)\times N$ matrix, with elements in $\R^d$:
\[
 (x_{ij}) = \begin{pmatrix} x_2 & x_3 & \cdots & x_{N+1} \\ x_1 & x_3 & \cdots & x_{N+1} \\ \vdots & \vdots & \ddots & \vdots \\ x_1 & x_2 & \cdots & x_N \end{pmatrix},
\]
where the $i$-th row is $(x_1, \dotsc, x_{i-1}, x_{i+1}, \dotsc, x_{N+1})$.
We want to construct a transport plan of the form
\[
 P = N \sum_{i = 1}^{N+1} a_i (\delta_{x_{i1}} \otimes \dotsb \otimes \delta_{x_{iN}})_{sym},
\]
where $a_i \geq 0$. Note that, by Lemma \ref{sym-marginals}, the marginals of $P$ are equal to
\[
 \pi(P) = \sum_{j = 1}^{N+1} \left( \sum_{\substack{ i = 1 \\ i \neq j}}^{N+1} a_i \right) \delta_{x_j}.
\]

Thus, the condition on the $a_i$'s to have $\pi(P) = \rho$ is
\[
 A_N \begin{pmatrix} a_1 \\ \vdots \\ a_{N+1} \end{pmatrix} = \begin{pmatrix} b_1 \\ \vdots \\ b_{N+1} \end{pmatrix},
\]
\ie,
\[
 \begin{pmatrix} a_1 \\ \vdots \\ a_{N+1} \end{pmatrix} = A_N^{-1} \begin{pmatrix} b_1 \\ \vdots \\ b_{N+1} \end{pmatrix}.
\]

Finally, observe that the condition \eqref{main-condition} implies that $a_1 \geq 0$, while the fact that $b_1 \geq b_2 \geq \dotsb \geq b_{N+1}$ leads to $a_1 \leq a_2 \leq \dotsb \leq a_{N+1}$, and hence $a_i \geq 0$ for every $i$ and we are done. 

\paragraph{Inductive step} Let $(b_1, \dotsc, b_k)$ satisfying \eqref{main-condition}, with $k \geq N+2$ (otherwise we are in the case $\mathfrak{P}(N,N+1)$, already proved). Take $t_2,\dotsc, t_k$ given by Lemma \ref{t-lemma}, and apply the inductive hypotheses to find
\begin{itemize}
 \item a symmetric transport plan $Q_1$ of finite cost in $(N-1)$ variables, with marginals
 \[
  \pi(Q_1) = \sum_{j=2}^k t_j \delta_{x_j}; 
 \]
 \item a symmetric transport plan $R$ of finite cost in $N$ variables, with marginals
 \[
  \pi(R) = \sum_{j=2}^k (b_j-t_j) \delta_{x_j}.
 \]
\end{itemize}

Define
\[
 Q = \frac{1}{N-1} \sum_{j = 1}^N (Q_1 \otimes_j \delta_{x_1}).
\]

Since $Q_1$ is symmetric, $Q$ is symmetric. Moreover, using Lemma \ref{t-lemma} (i),
\[
 \pi(Q) = \frac{1}{N-1} \delta_{x_1} \sum_{j=2}^k t_j + \sum_{j=2}^k t_j \delta_{x_j} = b_1 \delta_{x_1} + \sum_{j=2}^k t_j \delta_{x_j}.
\]

The transport plan $P = Q+R$ is symmetric, with marginals $\pi(P) = b_1 \delta_{x_1} + \dotsb + b_{k} \delta_{x_k}$.

\end{proof}

In order to conclude the proof of this Section, we must now deal not only with the non-atomic part of $\rho$, but also with the additional hypothesis of Proposition \ref{mega-prop}. Indeed, the presence of a non-atomic part will fix the atomic mass exceeding the inequality \eqref{mega-prop-condition}, as will be seen soon.

\begin{mydef} Given $N$, we say that the tuple $(b_1, \dotsc, b_\ell)$ is \emph{fast decreasing} if
 \[
  (N-j)b_j > \sum_{i > j} b_i \hsmallspace \forall j = 1, \dotsc, \ell-1.
 \]
\end{mydef}

\begin{remark} \label{maximal-fd-part} Note that if $(b_1, \dotsc, b_\ell)$ is fast decreasing, then necessarily $\ell < N$. As a consequence, given any sequence $(b_1, b_2, \dotsc )$, even infinite, we may select its maximal fast decreasing initial tuple $(b_1, \dotsc, b_\ell)$ (which might be empty, \ie, $\ell = 0$).
\end{remark}

\begin{thm} \label{thm-ell-atoms} If $\rho$ is such that
	\[\rho = \sigma + \sum_{j = 1}^k b_j \delta_{x_j}\]
	with $k > N$ atoms, then there exists a transport plan of finite cost.

\end{thm}

\begin{proof} Consider $(b_1, \dotsc, b_k)$ and use the Remark \ref{maximal-fd-part} to select its maximal fast decreasing initial tuple $(b_1, \dotsc, b_\ell)$, $\ell < N$. Thanks to Proposition \ref{mega-prop}, we may construct a transport plan $P_{\ell+1}$ over $\R^{(N-\ell)d}$ with marginals $b_{\ell+1} \delta_{x_{\ell+1}} + \dotsb + b_k\delta_{x_k}$, since
\[
 (N-\ell-1)b_{\ell+1} \leq \sum_{j = \ell+2}^k b_j
\]
by maximality of $(b_1, \dotsc, b_\ell)$ --- and this is condition \eqref{main-condition} in this case. We extend step by step $P_{\ell+1}$ to an $N$-transport plan, letting
\[
 P_{j} = \frac{1}{N-j} \sum_{i = j}^N (P_{j+1} \otimes_i \delta_{x_j}),
\]
for $j = \ell, \ell-1, \dotsc, 1$.

Let $p_\ell = b_{\ell+1} + \dotsb + b_{k}$, and $q_\ell = \tfrac{p_\ell}{N-\ell}$. We claim that $\abs{P_j} = (N-j+1) q_\ell$. In fact, by construction $\abs{P_{\ell+1}} = p_\ell$, and inductively
\[
 \abs{P_j} = \frac{1}{N-j} \sum_{i = j-1}^N \abs{P_{j+1}} = \frac{N-j+1}{N-j} (N-j) q_\ell = (N-j+1) q_\ell.
\]

Moreover,
\[
 \pi(P_j) = \sum_{i = j}^k q_\ell \delta_{x_i} + \sum_{i = \ell+1}^k b_i \delta_{x_i}.
\]

This is true by construction in the case $j = \ell+1$, and inductively
\[
 \pi(P_j) = \frac{1}{N-j} \delta_{x_j} \abs{P_{j+1}} + \frac{N-j}{N-j} \pi(P_{j+1}) = \sum_{i = j}^\ell q_\ell \delta_{x_i} + \sum_{i = \ell+1}^k b_i \delta_{x_i}.
\]

Note that, for every $i = 1, \dotsc, \ell$, $b_i \geq b_\ell > q_\ell$. We shall find, using Proposition \ref{partition-construction}, a transport plan of finite cost with marginals
\[
 \sigma + \sum_{i = 1}^\ell (b_i-q_\ell) \delta_{x_i},
\]

since the condition \eqref{existence-condition-eq2} reads
\[
 N(b_1 - q_\ell) - \sum_{i = 1}^\ell (b_i-q_\ell) = Nb_1 - \sum_{i = 1}^\ell b_i - (N-\ell) q_\ell < 1 - \sum_{i = 1}^k b_i = \abs{\sigma}.
\]

\end{proof}


\section{Marginals with countably many atoms} \label{final-section}

In this Section we finally deal with the case of an infinite number of atoms, \ie,
\[
 \rho = \sigma + \sum_{j = 1}^\infty b_j \delta_{x_j} 
\]
with $b_j > 0$, $b_{j+1} \leq b_{j}$ for every $j \geq 1$.

The main issue is of topological nature: if the atoms $x_j$ are too close each other (for example, if they form a dense subset of $\R^d$) and the growth of $b_j$ for $j \to \infty$ is too slow, the cost might diverge. With this in mind, we begin with an elementary topological result, in order to separate the atoms in $N$ groups, with controlled minimal distance from each other.

\begin{lemma} \label{set-partition}
There exists a partition $\R^d = E_2 \sqcup \dotsb \sqcup E_{N+1}$ such that:
\begin{enumerate}[(i)]
 \item for every $j=2, \dotsc, N+1$, $x_j \in \mathring{E}_j$;
 \item for every $j = 2, \dotsc, N+1$, $\partial E_j$ does not contain any $x_i$.
\end{enumerate}
\end{lemma}

\begin{proof}
 For $j = 3, \dotsc, N+1$ let $r_j > 0$ small enough such that
 \[
  x_i \notin B(x_j, r_j) \hsmallspace \text{for every $i = 1, \dotsc, N$, $i \neq j$.}
 \]
 
 Fixed any $j = 3, \dotsc, N+1$, by a cardinality argument there must be a positive real $t_j$ with $0 < t_j < r_j$ and $\partial B(x_j, t_j)$ not containing any $x_i$, $i \geq 1$. We take $E_j = B(x_j, t_j)$ for $j=3, \dotsc, N+1$. Note that this choice fullfills the conditions (i), (ii) for $j = 3, \dotsc, N+1$. Finally, we take
 \[
  E_2 = \R^d \setminus \left( \bigcup_{j = 3}^{N+1} E_j \right)
 \]
 Clearly $x_2 \in \mathring{E_2}$, and moreover the condition (ii) is satisfied, since
 \[
  \partial E_2 = \bigcup_{j=3}^{N+1} \partial E_j.
 \]
\end{proof}

Consider the partition given by Lemma \ref{set-partition}, and define the corresponding partition of $\N$ given by $\N = A_2 \cup \dotsb \cup A_{N+1}$, where
\[
 A_j = \gra{i \in \N \st x_i \in E_j}.
\]

Next we consider, for every $j = 2, \dotsc, {N+1}$ a threshold $n_j \geq 2$ large enough such that, defining
\[
 \epsilon_j = \sum_{\substack{ i \geq n_j \\ i \in A_j}} b_i,
\]
then
\begin{equation} \label{eps-condition}
 \epsilon_2 + \dotsb + \epsilon_{N+1} < \min \gra{ b_{N+1}, \frac{1}{N} - b_1}.
\end{equation}

This may be done since the series $\sum b_i$ converges, and hence for every $j = 2, \dotsc, {N+1}$ the series
\[
 \sum_{i \in A_j} b_i
\]
is convergent.

For every $j = 2, \dotsc, N+1$ define the following transport plan:
\[
 P_j = N \left[ \left(\sum_{i \in A_j, i \geq n_j} b_i \delta_{x_i} \right) \otimes \delta_{x_2} \otimes \dotsb \otimes \hat{\delta}_{x_j} \otimes \dotsb \otimes \delta_{x_{N+1}} \right]_{sym},
\]
and note that, by Lemma \ref{sym-marginals},
\[
 \pi(P_j) = \epsilon_j \sum_{\substack{h = 2 \\ h \neq j}}^{N+1} \delta_{x_h} + \sum_{\substack{ i \geq n_j \\ i \in A_j}} b_i \delta_{x_i}.
\]

Then let
\[
 P_{\infty} = \sum_{j = 2}^{N+1} P_j,
\]
and observe that
\[
 \pi(P_{\infty}) = \sum_{j = 2}^{N+1} \left( \sum_{\substack{i = 2 \\ i \neq j}}^{N+1} \epsilon_i \right) \delta_{x_j} + \sum_{j = 2}^{N+1} \sum_{\substack{ i \geq n_j \\ i \in A_j}} b_i \delta_{x_i}.
\]

Let now
\[
 \tilde{b}_i = \left\{
 \begin{array}{ll} \displaystyle
  b_i - \sum_{\substack{h = 2 \\ h \neq i}}^{N+1} \epsilon_h &\text{if $2 \leq i \leq N+1$} \\
  0 & \text{if $i \geq n_j$ and $i \in A_j$ for some $j = 2, \dotsc, N+1$} \\
  b_i & \text{otherwise}.
 \end{array} \right.
\]

We are left to find a transport plan of finite cost with marginals
\[
 \sigma + \sum_{i = 1}^\infty \tilde{b}_i \delta_{x_i},
\]
which has indeed a finite number of atoms. Note that $\tilde{b}_i \geq 0$ for every $i$, thanks to condition \eqref{eps-condition}. Moreover, since $\tilde{b}_1 = b_1$ and $\tilde{b}_j \leq b_j$, then $\tilde{b}_1 \geq \tilde{b}_j$ for every $j \in \N$, as is used in what follows. If
\[
 (N-1)\tilde{b}_1 \leq \sum_{i = 2}^\infty \tilde{b}_i
\]
we may conclude using Proposition \ref{mega-prop}. Otherwise, we proceed like in the proof of Theorem \ref{thm-ell-atoms}, with $\{ \tilde{b}_j \}$ replacing $\gra{b_j}$. At the final stage, it is left to check that
\[
 N(\tilde{b}_1 - \tilde{q}_{k+1}) - \sum_{i = 1}^k (\tilde{b}_i - \tilde{q}_{k+1}) < 1 - \sum_{i = 1}^\infty b_i = \abs{\sigma}.
\]

Indeed this is true, since using the condition \eqref{eps-condition} one gets
\[
 N(\tilde{b}_1 - \tilde{q}_{k+1}) - \sum_{i = 1}^k (\tilde{b}_i - \tilde{q}_{k+1}) = Nb_1 - \sum_{i = 1}^\infty b_i + N(\epsilon_2 + \dotsc + \epsilon_{N+1}) < 1 - \sum_{i = 1}^\infty b_i.
\]

\bibliographystyle{plain}

\bigskip
{\small \noindent
\textsc{Ugo Bindini}\\
Scuola Normale Superiore,\\
Piazza dei Cavalieri, 7\\
56127 Pisa - ITALY\\
{\tt ugo.bindini@sns.it}\\

\end{document}